\documentclass[11pt,a4paper]{article}
\usepackage[utf8]{inputenc}
\usepackage[T1]{fontenc}
\usepackage{lmodern,amsmath,amsthm,amsfonts,amssymb,graphicx,float,microtype,thmtools,mathtools,verbatim}
\usepackage{tocloft}
\usepackage[english]{babel}
\usepackage[shortlabels]{enumitem}
\setlist[itemize]{topsep=0ex,itemsep=0ex,parsep=0.3ex}
\setlist[enumerate]{topsep=0ex,itemsep=0ex,parsep=0.3ex}
\usepackage[dvipsnames,svgnames,table]{xcolor}
\usepackage[%
unicode=true,
colorlinks=true,
breaklinks=true,
linkcolor={blue!60!black},
citecolor={black},
urlcolor={blue!60!black},
pdftitle={3-Colouring Planar Graphs}]{hyperref}
\usepackage[capitalise, compress, nameinlink, noabbrev]{cleveref}
\crefname{lem}{Lemma}{Lemmas}
\crefname{thm}{Theorem}{Theorems}
\crefname{cor}{Corollary}{Corollaries}
\crefname{prop}{Proposition}{Propositions}
\crefformat{equation}{(#2#1#3)}
\Crefformat{equation}{Equation #2(#1)#3}
\setlist[enumerate,2]{label=(\roman*),ref=(\roman{enumi}.\roman*)}
\newcommand{\authorreveal}[1]{}
\newcommand{\defn}[1]{\textcolor{Maroon}{\emph{#1}}}

\usepackage[longnamesfirst,numbers,sort&compress]{natbib}
\makeatletter
\def\NAT@spacechar{~}
\makeatother
\usepackage[margin=30mm]{geometry}
\renewcommand{\baselinestretch}{1.05}
\DeclarePairedDelimiter{\ceil}{\lceil}{\rceil}

\renewcommand{\epsilon}{\varepsilon}
\renewcommand{\emptyset}{\varnothing}
\renewcommand{\ge}{\geqslant}
\renewcommand{\le}{\leqslant}
\renewcommand{\geq}{\geqslant}
\renewcommand{\leq}{\leqslant}

\DeclareMathOperator{\tw}{tw}
\DeclareMathOperator{\gm}{gm}

\newcommand{\PP}{\mathcal{P}}

\newcommand{\QQ}{\mathcal{Q}}

\newcommand{\NN}{\mathbb{N}}

\theoremstyle{plain}
\newtheorem{thm}{Theorem}
\newtheorem{lem}[thm]{Lemma}

\crefname{obs}{Observation}{Observations}

\newtheorem*{lem*}{lem}
\theoremstyle{definition}

\newtheorem*{conj*}{Conjecture}

\begin{document}
\title{\bf\boldmath\fontsize{18pt}{18pt}\selectfont
3-Colouring Planar Graphs}

\author{%
Vida Dujmovi{\'c}\,\footnotemark[3]\qquad
Pat Morin\,\footnotemark[5] \qquad
Sergey Norin\,\footnotemark[4] \qquad
David~R.~Wood\,\footnotemark[2]
}

\maketitle

\begin{abstract}
We show that every $n$-vertex planar graph is 3-colourable with monochromatic components of size $O(n^{4/9})$. The best previous bound was $O(n^{1/2})$ due to Linial, Matou{\v{s}}ek, Sheffet and Tardos [\emph{Combin. Probab. Comput.}, 2008].
\end{abstract}

\renewcommand{\thefootnote}{\fnsymbol{footnote}}

\footnotetext[3]{School of Computer Science and Electrical Engineering, University of Ottawa, Ottawa, Canada (\texttt{vida.dujmovic@uottawa.ca}). Research supported by NSERC and the Australian Research Council.}

\footnotetext[4]{Department of Mathematics and Statistics, McGill University, Montr\'eal, Canada (\texttt{snorin@math.mcgill.ca}).  Research supported by NSERC.}

\footnotetext[5]{School of Computer Science, Carleton University, Ottawa, Canada (\texttt{morin@scs.carleton.ca}). Research  supported by NSERC.}

\footnotetext[2]{School of Mathematics, Monash University, Melbourne, Australia (\texttt{david.wood@monash.edu}). Research supported by the Australian Research Council and NSERC.}

\renewcommand{\thefootnote}{\arabic{footnote}}

\bigskip
\bigskip
\begin{center}
\includegraphics[width=120mm]{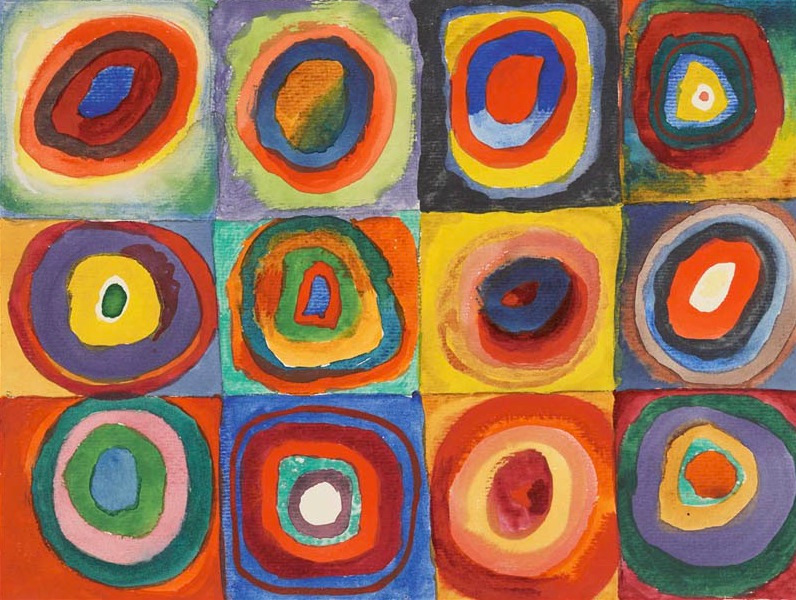}

``Color Study: Squares with Concentric Circles''
by Wassily Kandinsky, 1913

\end{center}

\newpage

\section{Introduction}
\label{Introduction}

The 4-Colour Theorem~\citep{AH89,RSST97} says that every planar graph is properly 4-colourable. What can be said if only three colours are available? Clustered colouring provides an avenue for addressing this question. 

A \defn{colouring} of a graph $G$ is a function that assigns one colour to each vertex of $G$. For an integer $k\geq 1$, a \defn{$k$-colouring} is a colouring using at most $k$ colours. A colouring of a graph is \defn{proper} if each pair of adjacent vertices receives distinct colours. A \defn{monochromatic component} with respect to a colouring of a graph $G$ is a connected component of the subgraph of $G$ induced by all the vertices assigned a single colour. A colouring has \defn{clustering} $c$ if every monochromatic component has at most $c$ vertices. Note that a colouring with clustering 1 is precisely a proper colouring. See \citep{WoodSurvey} for an extensive survey on this topic, and see~\citep{MRW17,Liu24,LMST08,KM07,vdHW18,CL25,LW1,LW2,LW3,LW4,LW24,Kawa08,DEMW23,KM07,EKKOS15,LO18,KO19,DN17,NSSW19,NSW22,BHW,KMRV97,DEMWW22,HW19,HKOSW24} for results on clustered colouring of planar graphs and other minor-closed classes. 

For integers $n,c\geq 1$, let $f_c(n)$ be the minimum integer such that every $n$-vertex planar graph is $c$-colourable with clustering $f_c(n)$. 
By the 4-Colour Theorem, $f_c(n)=1$ for $c\geq 4$. \citet{LMST08} showed that $f_2(n)\in\Theta(n^{2/3})$. The $c=3$ case is wide open. The best known lower bound\footnote{This lower bound is established as follows. Fix an integer $k\geqslant 2$. Let $F_1,\dots,F_k$ be disjoint fans, each with $k^2$ vertices. Let $w_i$ be the centre of $F_i$. Let $G$ be obtained by adding one vertex $v$ adjacent to every vertex in $F_1\cup\dots\cup F_k$. Observe that $G$ is planar with $k^3+1$ vertices. Suppose that $G$ is 3-colourable with clustering less than $k$. Say $v$ is red. Since $v$ is adjacent to every other vertex, some $F_i$ has no red vertex. Say $w_i$ is blue. So the path $F_i-w_i$ has at most $k-1$ blue vertices, and no red vertices. Deleting the blue vertices from $F_i-w_i$ leaves at most $k$ components, all of which are yellow. At least one such component has at least $\ceil{(k^2-(k-1))/k}=k$ vertices, which is a contradiction. So every 3-colouring has clustering at least $k\approx n^{1/3}$.} is $f_3(n)\in \Omega(n^{1/3})$, implicitly proven by \citet*{KMRV97}. The best known upper bound is $f_3(n)\in O(n^{1/2})$, due to \citet*{LMST08}. This paper improves the upper bound on $f_3(n)$ as follows. 



\begin{thm}
\label{3ColourPlanar}
Every $n$-vertex planar graph is 3-colourable with clustering $O(n^{4/9})$.
\end{thm}

The proof of \cref{3ColourPlanar} depends on new results regarding treewidth and separators in planar graphs that are of independent interest. This material is introduced in \cref{TreewidthSeparators}. The proof of \cref{3ColourPlanar} is presented in \cref{MainProof}. All our proofs give explicit (though non-optimised) constants.


\section{Treewidth and Separators}
\label{TreewidthSeparators}

We consider finite simple graphs $G$ with vertex-set $V(G)$ and edge-set $E(G)$. 
A \defn{tree-decomposition} of a graph $G$ is a collection $(B_x :x\in V(T))$ of subsets of $V(G)$ (called \defn{bags}) indexed by the vertices of a tree $T$, such that (a) for every edge $uv\in E(G)$, some bag $B_x$ contains both $u$ and $v$, and (b) for every vertex $v\in V(G)$, the set $\{x\in V(T):v\in B_x\}$ induces a non-empty (connected) subtree of $T$. 
The \defn{width} of $(B_x:x\in V(T))$ is $\max\{|B_x| \colon x\in V(T)\}-1$. 
The \defn{treewidth} of a graph $G$, denoted by \defn{$\tw(G)$}, is the minimum width of a tree-decomposition of $G$. Treewidth is the standard measure of how similar a graph is to a tree; see \citep{HW17,Bodlaender98,Reed97} for surveys.

A \defn{balanced separator} in a graph $G$ is a set $S\subseteq V(G)$ such that every component of $G-S$ has at most $\frac12|V(G)|$ vertices. Treewidth and balanced separators are closely related. In particular, \citet{RS-II} showed that every graph $G$ has a balanced separator of size at most $\tw(G)+1$. Conversely, a result of \citet{DN19} implies that if every subgraph of a graph $G$ has a balanced separator of size at most $k$, then $\tw(G)\leq 15k$ (also see \citep{HMM25}). 

The next definition generalises the notion of a balanced separator. For $q\in\mathbb{R}^+$, a set $S$ of vertices in a graph $G$ is a \defn{$q$-separator} if every component of $G-S$ has at most $q$ vertices. \citet{DvoWoo} proved the following qualitative generalisation of the result of \citet{RS-II} mentioned above. 


\begin{lem}[{\protect\citep[Lemma~25]{DvoWoo}}]
\label{TreewidthSep}
Let $k,n\in\NN$ and $p,q\in\mathbb{R}_{>0}$ with $n(k+1) \leq pq$ and $p\geq k+1$. Then any $n$-vertex graph $G$ with $\tw(G)\leq k$ has a $q$-separator of size at most $p$.
\end{lem}


Clustered colourings and $q$-separators are inherently related, since a $k$-colouring of a graph has clustering $c$ if and only if the union of any $k-1$ colour classes is a $c$-separator.

\subsection{Planar Graphs}

\citet{LT79} showed that every $n$-vertex planar graph has a balanced separator of size $O(\sqrt{n})$, and thus has treewidth $O(\sqrt{n})$. (Tree-decompositions of $n$-vertex planar graphs with width $O(\sqrt{n})$ can be directly constructed without using the above-mentioned result of \citet{DN19}; see \citep{DMW17} for example.)\ 

A graph $H$ is a \defn{minor} of a graph $G$ if $H$ is isomorphic to a graph that can be obtained from a subgraph of $G$ by contracting edges. A graph~$G$ is \defn{$H$-minor-free} if~$H$ is not a minor of~$G$. A \defn{model} of a graph $H$ in a graph $G$ is a set $\{G_x:x\in V(H)\}$ of pairwise vertex-disjoint connected subgraphs of $G$, such that for each $xy\in E(H)$ there is an edge of $G$ between $G_x$ and $G_y$. Each subgraph $G_x$ is called a \defn{branch set} of the model. Observe that $H$ is a minor of $G$ if and only if there is a model of $H$ in $G$. 

The \defn{$k\times k$ grid} is the graph with vertex-set $\{1,\dots,k\}^2$ where vertices $(x,y)$ and $(x',y')$ are adjacent if and only if $|x-x'|+|y-y'|=1$. For $k\geq 2$, the $k\times k$ grid has treewidth $k$ 
(see \citep{HW17} for example). For a graph $G$, let \defn{$\gm(G)$} be the maximum $k$ such that the $k\times k$ grid is a minor of $G$. Since treewidth is minor-monotone, $\tw(G)\geq \gm(G)$. 
Conversely, \citet{RS-V} showed there is a function $f$ such that $\tw(G)\leq f(\gm(G))$ for every graph $G$.
For planar graphs, \citet{RST94} showed that $f$ can be taken to be linear. 

\begin{lem}[\citep{RST94}]  
\label{PlanarGridMinor}
Every planar graph $G$ satisfies $\tw(G) \leq 6\gm(G)+1$. 
\end{lem}

\cref{PlanarGridMinor} can be considered a qualitative strengthening of the Lipton-Tarjan treewidth bound, since 
$n\geq \gm(G)^2$ for any $n$-vertex graph $G$, so if $G$ is planar, then \cref{PlanarGridMinor} implies $\tw(G)\leq 6\gm(G)+1\leq 6\sqrt{n}+1$.   


\citet{LT80} used their above-mentioned result for balanced separators to show that every $n$-vertex planar graph has a $q$-separator of size $O(\frac{n}{q^{1/2}})$. For the sake of completeness, we include the proof with an explicit constant. 

\begin{lem}
\label{LT} 
For any $q>0$ any $n$-vertex planar graph $G$ has a $q$-separator of size at most $\frac{12n}{q^{1/2}}$.
\end{lem}

\begin{proof}
\citet[(2.1)]{AST-SJDM94} showed that every $n$-vertex planar graph has a $\frac{2n}{3}$-separator of size at most $c\sqrt{n}$, where $c=\frac{3}{\sqrt{2}}$. Run the following algorithm. Initialise $S:=\emptyset$. While $G-S$ has a component $X$ with more than $q$ vertices, let $S_X$ be a $\frac{2|V(X)|}{3}$-separator of $X$ with size at most $c\sqrt{|V(X)|}$, and add $S_X$ to $S$. 

At the end of the algorithm, $S$ is a $q$-separator. Say a component of $G-S$ has \emph{level} 0. Say $X$ is a component of $G-S$ at some stage of the algorithm, but $X$ is not a component of $G-S$ at the end of the algorithm. Then $X$ is separated by some set $S_X$, which is then added to $S$. Define the \emph{level} of $X$ to 1 plus the maximum level of a component of $X-S_X$. 

By assumption, level 0 components have at most $q$ vertices.  Each level 1 component has more than $q$ vertices.  By induction on $i$, each level $i\geq 1$ component has more than $(\frac{3}{2})^{i-1} q$ vertices. Let $t_i$ be the number of components at level $i\geq 1$. 
Say $X_1,\dots,X_{t_i}$ are the components at level $i$. 
Since level $i$ components are pairwise disjoint, 
$$t_i (\tfrac32)^{i-1} q < \sum_{j=1}^{t_i}|V(X_j)| \leq n. $$ 
The number of vertices added to $S$ by separating level $i$ components is at most 
$$c\sum_{j=1}^{t_i} |V(X_j)|^{1/2}
\leq c \left(t_i \sum_{j=1}^{t_i}|V(X_j)|\right)^{1/2} \!\!\!\!\!
\leq c (t_i n)^{1/2}
< c\left(\frac{n}{ (\frac32)^{i-1} q} \right)^{1/2} \!\!\!\!\! n^{1/2}
= cn  \left(\frac{(\frac32)^{1-i}}{q} \right)^{1/2},$$
where the first inequality uses the Cauchy-Schwarz inequality.
Hence 
\begin{equation*}
|S| 
\leq \sum_{i\geq 1} cn \left(\frac{(\frac32)^{1-i}}{q} \right)^{1/2}
= \frac{cn}{q^{1/2}} \sum_{i\geq 1}  (\tfrac32)^{(1-i)/2}
= \frac{c\,(\frac32)^{1/2} n}{((\frac32)^{1/2}-1) q^{1/2}}
< \frac{12n}{q^{1/2}}.\qedhere
\end{equation*}
\end{proof}


Note that \cref{LT} establishes the previous upper bounds on $f_2(n)$ and $f_3(n)$ by \linebreak \citet{LMST08} mentioned in \cref{Introduction}. In the case of two colours, by \cref{LT}, $G$ has a $O(n^{2/3})$-separator $S$ of size  $O(n/(n^{2/3})^{1/2})=O(n^{2/3})$. Colour $S$ blue and $V(G)\setminus S$ red, to obtain a 2-colouring with clustering $O(n^{2/3})$. In the case of three colours, by \cref{LT}, $G$ has a $O(n^{1/2})$-separator $S$ of size $O(n^{3/4})$.  By \cref{LT} again, $G[S]$ has a $O(n^{1/2})$-separator $S_1$ of size $O( n^{3/4} / (n^{1/2})^{1/2})=O(n^{1/2})$. Colour $V(G)\setminus S$ red, colour $S\setminus S_1$ blue, and colour $S_1$ yellow, to obtain a 3-colouring with clustering $O(n^{1/2})$.




As noted by \citet{LMST08}, these proofs work in any hereditary class admitting $O(\sqrt{n})$ balanced separators, which includes all proper minor-closed classes~\citep{AST90}. To improve the bound for planar graphs, we use a more nuanced version of the above approach based on specific properties of planar graphs. 
Separators again play a key role. In addition to considering the size of a $q$-separator, it will be important to bound the treewidth of the subgraph induced by a $q$-separator. Say $S$ is a $q$-separator in an $n$-vertex planar graph $G$ with $|S|\leq\frac{12n}{q^{1/2}}$. By the Lipton-Tarjan treewidth bound, $G[S]$ has treewidth $O(\sqrt{|S|})$, which is $O(\frac{n^{1/2}}{q^{1/4}})$. At the heart of our proof is the observation that $G$ has a $q$-separator $S$ such that $G[S]$ has treewidth significantly less than this naive bound, as shown in the following lemma.

\begin{lem}
\label{SmallCompsSmallTreewidth} 
For any $q>0$ every $n$-vertex planar graph $G$ has a $q$-separator $S\subseteq V(G)$ such that $|S|\leq \frac{12n}{q^{1/2}}$ and $\gm(G[S])< 2(\frac{n}{q})^{1/2}+2$ and $\tw(G[S])< 12(\frac{n}{q})^{1/2}+13$.
\end{lem}


\begin{proof}
  Let $S\subseteq V(G)$ be a $q$-separator in $G$ of minimum size. By \cref{LT}, $|S|\leq \frac{12n}{q^{1/2}}$. Let $k$ be the integer such that $\gm(G[S])\in\{2k+1,2k+2\}$.  Let $(X_{i,j}:i,j\in\{1,\dots,2k+1\})$ be a model of the $(2k+1)\times(2k+1)$ grid in $G$. Let $M:=\bigcup_{i,j} V(X_{i,j})$. So each $X_{i,j}$ is a connected subgraph of $G[S]$, such that contracting each $X_{i,j}$ to a vertex, and deleting $S\setminus V(M)$ gives a $(2k+1)\times(2k+1)$ grid (possibly plus extra edges). For $i,j\in\{1,\dots,2k+1\}$, let $C_{i,j}$ be the union of the components of $G-S$ that are adjacent to $X_{i,j}$. By the minimality of $S$, $|V(C_{i,j}\cup X_{i,j})|>q$ (otherwise $S-V(X_{i,j})$ would be a $q$-separator of $G$, contradicting the minimality of $S$). 
  By planarity, if $2\leq i,j,i',j'\leq 2k$, and 
  $|i-i'|\geq 2$ or $|j-j'|\geq 2$, then $C_{i,j}\cap C_{i',j'}=\emptyset$. 
  So $C_{2i,2j} \cap C_{2i',2j'}=\emptyset$ whenever $1\leq i,j,i',j'\leq k$ and $(i,j)\neq(i',j')$. Thus 
  $k^2 q< n$ and 
  $k < (\frac{n}{q})^{1/2}$. 
Therefore, $\gm(G[S])\leq 2k+2 
< 2(\frac{n}{q})^{1/2}+2$.
By \cref{PlanarGridMinor}, 
$\tw(G[S])< 12(\frac{n}{q})^{1/2}+13$.
\end{proof}


As an aside, we now show that \cref{SmallCompsSmallTreewidth} is best possible. For simplicity, assume that $n^{1/2}$ and $q^{1/2}$ are integers. Let $G$ be the $n^{1/2}\times n^{1/2}$ grid. Let $S$ be the union of every $q^{1/2}$-th row and column in $G$. So $S$ is a $q$-separator for $G$, since each component of $G-S$ is a subgraph of a $q^{1/2} \times q^{1/2}$ grid. Note that $S$ consists of roughly $\frac{2n^{1/2}}{q^{1/2}}$ rows and columns, each of size $n^{1/2}$. So $|S|\approx \frac{2n}{q^{1/2}}$, which matches the bound in \cref{SmallCompsSmallTreewidth} up to a constant factor. The subgraph $G[S]$ is a subdivision of a $\frac{n^{1/2}}{q^{1/2}}\times \frac{n^{1/2}}{q^{1/2}}$ grid, implying $\tw(G[S])=\gm(G[S])\approx (\frac{n}{q})^{1/2}$, again matching the bound in \cref{SmallCompsSmallTreewidth} up to a constant factor.

We actually use a variant of \cref{SmallCompsSmallTreewidth} in our proof of \cref{3ColourPlanar}. We highlight  \cref{SmallCompsSmallTreewidth} here since it is of independent interest. 

\section{Main Proof}
\label{MainProof}

A graph embedded in $\mathbb{R}^2$ with no edge crossings is a \defn{plane graph}. A plane graph is a \defn{triangulation} if each face is bounded by a cycle of three edges. For a plane graph $G$, let $F(G)$ denote the set of faces of $G$.  For a vertex $v$ of $G$, let $F(G,v)$ be the set of faces of $G$ incident to $v$. For $A\subseteq F(G)$, a \defn{face-weighting} is any function $w:A\rightarrow\mathbb{R}_{\geq 0}$. For $B\subseteq A$, let $w(B):=\sum_{f\in B}w(f)$.

\begin{lem}\label{tw_vs_weight}
    For any real numbers $N\ge W\ge 1$, for any plane graph $G$, for any set $X\subseteq F(G)$, and for any face-weighting $w:F(G)\setminus X\to\mathbb{R}_{\geq 0}$ such that $w(F(G)\setminus X) \le N$, and $F(G,v)\cap X\neq\emptyset$ or $w(F(G,v)) \ge W$ for each $v\in V(G)$,
    $$\tw(G)\le 12\sqrt{|X|+N/W}+7.$$
\end{lem}



\begin{proof}
Extend the face-weighting $w$ so that, for any subgraph $Q$ of $G$ and each face $f$ of $Q$, $w(f):=\sum_{f'} w(f')$, where $f'$ ranges over all faces in $F(G) \setminus X$ contained in $f$. Let $k$ be the maximum integer such that the $k\times k$ grid $H$ is a minor of $G$. By \cref{PlanarGridMinor}, $\tw(G)\leq 6(k+1)-5=6k+1$.     Let $G'$ be a minimal subgraph of $G$ that contains $H$ as a minor. Let $(G'_x:x\in V(H))$ be a model of $H$ in $G'$.

By minimality, for each $x\in V(H)$, the subgraph $G'_x$  is a tree. Thus, there is a 1--1 correspondence between faces of $G'$ and faces of $H$. For every `internal' face $f$ of $G'$, send a `charge' of $\frac{w(f)}{4}$ to each of the four vertices of $H$ incident to the face of $H$ corresponding to $f$. The total charge is at most the total weight, which is at most $N$. 

Say an internal face of $H$ is \defn{special} if the corresponding face of $G'$ contains a face in $X$. At most $4|X|$ vertices of $H$ are incident to a special face. Consider an internal vertex $x$ of $H$ incident to no special face. Let $v$ be any vertex in $G'_x$. So $F(G,v)\cap X=\emptyset$, implying $w(F(G,v)) \geq W$. Since $x$ is internal, each face in $F(G,v)$ is contained within an internal face of $G'$. So $x$ receives a charge of at least $\tfrac{w(F(G,v))}{4} \geq \frac{W}{4}$. Thus, at most $4N/W$ internal vertices of $H$ are incident to no special face. Hence, the number of internal vertices in $H$, namely $(k-1)^2$, is at most $4|X| + 4N/W$. Therefore  $k\leq 2\sqrt{|X|+N/W}+1$ and $\tw(G)\le 12\sqrt{|X|+N/W}+7$.
\end{proof}

The next lemma is implicitly used in (6.3) by \citet{RST94}, who say the proof is a `straightforward modification' of the proof of (4.1) by \citet{RS-III}. We include a proof for completeness.

\begin{lem}
\label{FindGridMinor}
Let $\gamma$ be a cycle bounding the infinite face in a plane graph $G$ with $|V(\gamma)|=4t$. Let $v_1,\ldots,v_{4t}$ be the vertex set of $\gamma$ in order, let $A'=\{v_1,\ldots,v_t\}$ $A=\{v_{t+1},\ldots,v_{2t}\}$, $B'=\{v_{2t+1},\ldots,v_{3t}\}$, and $B=\{v_{3t+1}, \ldots, v_{4t}\}$. If $G$ contains $t$ pairwise vertex disjoint $(A,B)$-paths and $t$ pairwise vertex disjoint $(A',B')$-paths then $G$ contains a $t \times t$ grid minor.
\end{lem}

\begin{proof} Let $\PP=\{P_1,\ldots, P_t\}$ and $\QQ=\{Q_1,\ldots, Q_t\}$ be two sets of  pairwise vertex disjoint $(A,B)$- and $(A',B')$-paths, respectively. As $G$ is plane with boundary $\gamma$, we  assume without loss generality that $P_i$ has ends $v_{t+i}$ and $v_{4t-i+1}$ and $Q_i$ has ends $v_i$ and $v_{3t+1-i}$ for each $i \in [t]$.  Moreover, we may assume that $P_1,P_t,Q_1, Q_t \subseteq \gamma$, e.g. $V(P_1)= \{v_{t+1},v_t,\ldots,v_1,v_{4t}\}$.
		
 Suppose for a contradiction that $G$ does not contain a $t \times t$ grid minor, and choose such a graph $G$ with $|E(G)|$ minimum. In particular, $G$
 is minor-minimal subject to existence of paths  $\PP$  and $\QQ$ as above.  
		 
It follows that every edge of $G$ belongs to $\PP$ or $\QQ$ and every vertex in $V(G)$ belongs to both $\PP$ and $\QQ$, as otherwise we can contract any edge incident to it. If there exists $e \in E(\PP) \cap E(\QQ)$ we can also contract it and so we assume that there are no such edges, subject to the following adjustment of the setting. Currently, $v_{4t}v_1 \in  E(P_1) \cap E(Q_1)$ and contracting it reduces the length of $\gamma$, thus from now on we assume $v_{4t}=v_1$ and, similarly, $v_{at}=v_{at+1}$ for $a=1,2,3$. 
	
For $v \in V(G)$ define $i(v)$ and $j(v)$ to be the unique indices such that $v \in V(P_{i(v)}) \cap V(Q_{j(v)})$. Thus for every edge $e=uv$ we have $i(u)=i(v)$ or $j(u)=j(v)$, but not both. Indeed, if $i(u)=i(v)$ and $j(u)=j(v)$ and, say, $e \in E(\PP)$, without loss of generality, then we can reroute the subpath of $Q_{j(v)}$ between $u$ and $v$ along $e$, contradicting the choice of $G$. Additionally, using the fact that $G$ is plane we conclude that  if $i(u)=i(v)$ (i.e. $e \in E(\PP)$) then $|j(u)-j(v)| = 1$.
	
Consider $v \in V(G) \setminus V(\gamma)$. Let $u$  and $w$ be the neighbours of $v$ along $P_{i(v)}$ such that  $u, v$ and $w$ occur on $P_{i(v)}$ in this order as we traverse it from $B$ to $A$. We say that $v$ is a \emph{crater} if $j(u)=j(w)=j(v)-1$. If no vertex in $V(G) - V(\gamma)$ is a crater, then the function $j(\cdot)$ increases along every path in $\PP$ as we traverse it from $B$ to $A$ from $1$ to $t$. Thus every path in $\PP$ has $t-1$ edges and shares exactly one vertex with every path in $\QQ$. Thus $G$ is isomorphic to a $t \times t$ grid, which is the desired contradiction.
	
Thus we assume that there exists a crater $v \in V(G) \setminus V(\gamma)$ and choose such $v$ with $j(v)$ minimum. Let $i=i(v)$, $j=j(v)$, for brevity, and let $u, w \in V(P_i) \cap V(Q_{j-1})$ be the neighbours of $v$ on $P_i$. Then $u$ and $w$ are not adjacent on $Q_{j-1}$, thus there exists $v' \in Q_{j-1}$ between $u$ and $w$. Let $i'=i(v')$, the edges of $P_{i'}$ incident to $v'$ both have their second ends on $Q_j$ or $Q_{j-2}$, as observed earlier. However, neither can have an end on $Q_j$, as this would create a cross in $G$. Thus both of these edges have the second end on $Q_{j-2}$ and $v'$ is a crater, contradicting the choice of $v$. 
\end{proof}	

We make use of the following lemma: 

\begin{lem}
\label{cycle_separator}
For any plane triangulation $G$ and face-weighting  $w:F(G)\to\mathbb{R}_{\ge 0}$ with total weight $N:=w(F(G))$, there is a cycle $\lambda$ in $G$ of length at most $4\gm(G)+4$ such that $w( F(C)\cap F(G) )\le \frac23 N$ for each component $C$ of $G-V(\lambda)$. 
\end{lem}

\begin{proof}
 The following is an adaptation of the proof of the Planar  Cycle Separator Theorem given by \citet{AST-SJDM94}. First, observe that if $G$ contains a face $f^\star=xyz$ with $w(f^\star)\ge \tfrac13 N$, then the cycle $\lambda:=x,y,z$ satisfies the requirements of the lemma, since for any component $C$ of $G-\{x,y,z\}$, $$w(F(G)\cap F(C)) \leq w(F(G)\cap F(G-\{x,y,z\})) 
 \leq N-w(f^\star)\le \tfrac23 N.$$  
 Now assume that $w(f)< \tfrac13 N$ for each $f\in F(G)$.

  For any cycle $\lambda$ in $G$, let $w(\lambda^-)$ be the sum of $w(f)$ over all faces of $G$ contained in the interior of $\lambda$.  Let $w(\lambda^+)$ be the sum of $w(f)$ over all faces of $G$ contained in the exterior of $\lambda$.  Since each face  $f\in F(G)$ contributes $w(f)$ to exactly one of these two sums,  $w(\lambda^-)+w(\lambda^+)= w(F(G))= N$, for any cycle $\lambda$ in $G$.    If $w(\lambda^-)>w(\lambda^+)$, then define $\zeta(\lambda)$ to be the number of faces of $G$ contained in the interior of $\lambda$, otherwise define $\zeta(\lambda)$ to be the number of faces of $G$ contained in the exterior of $\lambda$.  Define 
  $\phi(\lambda):=(\max\{w(\lambda^-),w(\lambda^+)\}, \zeta(\lambda))$.
  Let $t:=\gm(G)+1$.  Let $\lambda$ be a cycle in $G$ of length at most $4t$ that lexicographically minimizes $\phi(\lambda)$.  We claim that $\lambda$ satisfies the conditions of the lemma.  

  
  If $\max\{w(\lambda^+),w(\lambda^-)\}\le \frac23 N$, then $\lambda$  satisfies the conditions of the lemma since, for any component $C$ of $G-V(\lambda)$, the faces of $F(G)\cap F(C)$ are either completely contained in the interior of $\lambda$ (so $w(F(G)\cap F(C)) \le w(\lambda^-)$) or completely contained in the exterior of $\lambda$ (so $w(F(G)\cap F(C)) \leq w(\lambda^+)$).  Now assume, without loss of generality, that $w(\lambda^-)>\frac23 N$. This assumption implies that $N-w(\lambda^-)=w(\lambda^+)<\tfrac13 N$ and that $\zeta(\lambda)$ is equal to the number of faces of $G$ contained in the interior of $\lambda$.  
    
If the interior of $\lambda$ contains no vertices, then every component $C$ of $G - V(\lambda)$ is contained in the exterior of $\lambda$, and so $w(F(C) \cap F(G)) < \tfrac{1}{3}N$ for every component of $G$.

Thus we assume that the interior of $\lambda$ contains at least one vertex.
We claim that $\lambda$ has length exactly $4t$.  Suppose, instead, that $\lambda$ has length at most $4t-1$. Let $xyz$ be any face of $G$ such that  $xz$ is an edge of $\lambda$ and $y$ is in the interior of $\lambda$.  (Such a face exists because $\lambda$ has at least one vertex of $G$ in its interior.)\  Let $\gamma$ be the cycle obtained from $\lambda$ by replacing the edge $xz$ with the path $xyz$.  Then $\gamma$ has length at most $4t$, $w(\gamma^-) = w(\lambda^-)-w(xyz)\le w(\lambda^-)$ and $w(\gamma^+)=w(\lambda^+)+w(xyz)<\frac23 N<w(\lambda^-)$. Furthermore, if $w(\gamma^-)=w(\lambda^-)$ (which happens when $w(xyz)=0$), then $\zeta(\gamma^-)=\zeta(\lambda^-)-1$, since the face $xyz$ is in the interior of $\lambda$ but not in the interior of $\gamma$. Therefore, $\phi(\gamma)<\phi(\lambda)$,  contradicting the choice of $\lambda$. Hence,  $\lambda$ has length exactly $4t$. 

Let $G^-$ be the subgraph of $G$ induced by $V(\lambda)$ and the vertices of $G$ in the interior of $\lambda$.  Let $v_0,\ldots,v_k$ be a path in the cycle $\lambda$ with $k<2t$. We claim that $v_0,\ldots,v_k$ is the unique shortest path from $v_0$ to $v_k$ in $G^-$.  Suppose that this were not the case and let $k$ be the minimum value for which $\lambda$ contains a path $v_0,\ldots,v_k$ but $G^-$ contains another path $P$ from $v_0$ to $v_k$ of length at most $k$.  By the minimality of $k$, $P$ has no vertices of $\lambda$ in its interior.  Therefore, the graph $\lambda\cup P$ contains three cycles, $\lambda$, $\lambda_1$ and $\lambda_2$ each having length at most $4t$.  Then $w(\lambda_1^-)+w(\lambda_2^-)=w(\lambda^-)$. Assume, without loss of generality, that $w(\lambda^-_1)\ge \tfrac{1}{2}w(\lambda^-)> \tfrac13 N$.  Then $w(\lambda_1^-)\le w(\lambda^-)$ and $w(\lambda_1^+)=N-w(\lambda_1^-)<\frac23 N<w(\lambda^-)$.  Furthermore, if $w(\lambda_1^-)=w(\lambda^-)$ (which happens when $w(\lambda_2^-)=0$), then $\zeta(\lambda_1)<\zeta(\lambda)$ since the faces of $G^-$ in the interior of $\lambda_2$ are in the interior of $\lambda$ but not in the interior of $\lambda_1$.
  Therefore, $\phi(\lambda_1)<\phi(\lambda)$,  contradicting the choice of $\lambda$. 
 
  Let $v_1,\ldots,v_{4t}$ be the vertices of $\lambda$, in order.  Let $X\subseteq V(G^-)$ be an inclusion-minimal set that separates $A:=\{v_{t+1},\ldots,v_{2t}\}$ and $B:=\{v_{3t+1},\ldots,v_{4t}\}$ in $G^-$.  That is, no component of $G^--X$ contains vertices from both $A$ and $B$. Then $X$ must contain at least one vertex from the path  $L:=v_{4t},v_1,\ldots,v_{t+1}$ and at least one vertex from the path $R:=v_{2t},\ldots,v_{3t+1}$.  Since $G^-$ is a near-triangulation and $X$ is a minimal separator of $A$ and $B$, $G^-[X]$ is connected\footnote{This follows from the fact that if $A$ and $B$ are disjoint connected subgraphs in a planar triangulation $H$, then every minimal separator for $A$ and $B$ is a cycle in $H$ (see \citep[Proposition~8.2.3]{MoharThom}). Apply this fact to the plane triangulation obtained from $G^-$ by adding a new vertex in the outerface adjacent to every vertex in $\lambda$.}. Therefore, $G^-[X]$ contains a path from a vertex in $L$ to a vertex in $R$.  By the preceding paragraph, this path has length at least $t-1$ and thus contains at least $t$ vertices.  Hence $|X|\ge t$.  Therefore, any set $X$ that separates $A$ and $B$ in $G^-$ has size at least $t$. By Menger's Theorem, $G^-$ contains $t$ pairwise vertex-disjoint paths $P_1,\dots,P_t$ each with one endpoint in $A$ and one endpoint in $B$. By the same argument (with indices shifted by $t$), $G^-$ contains $t$ pairwise vertex-disjoint paths $Q_1,\dots,Q_t$ between $A':=\{v_{1},\ldots,v_{t}\}$ and $B':=\{v_{2t+1},\ldots,v_{3t}\}$. By construction, each $P_i$ intersects each $Q_j$. 
  Thus $\gm(G)\geq t$ by \cref{FindGridMinor}.
\end{proof}

A \defn{noose} of a plane graph $G$ is a simple closed curve $\lambda$ in $\mathbb{R}^2$ such that for each edge $vw$ of $G$, $\lambda$ contains the entire edge $vw$ (including the vertices $v$ and $w$) or does not intersect the interior of $vw$. A graph $G$ is \defn{chordal} if every induced cycle has length at least 4. 


\begin{lem}
    \label{embedding_preserving_triangulation}
    For every plane graph $G$ with at least three vertices, edges can be added to $G$ to create an embedding-preserving plane triangulation $G'$ with $\tw(G')\leq\max\{\tw(G),3\}$.
\end{lem}

\begin{proof}
If $|V(G)|\leq 4$ then the complete graph $G'$ with vertex set $V(G)$ satisfies the claim, since $\tw(G')=|V(G)|-1\leq 3$. Now assume that $|V(G)|\geq 5$. It suffices to show that if $G$ is not a plane triangulation, then there exists two non-adjacent vertices $v,w\in V(G)$ on the same face of $G$,  such that $\tw(G+vw)\leq\max\{\tw(G),3\}$. (Here $G+vw$ is the graph obtained from $G$ by adding the edge $vw$ between non-adjacent vertices $v,w\in V(G)$.)\ 
Then repeat until we obtain a plane triangulation preserving the embedding of $G$. 
  
First suppose that $G$ is disconnected. Let $G_1,\dots,G_c$ be the components of $G$ where $c\geq 2$. So $\tw(G_i)\leq\tw(G)$. Let $(B_x:x\in V(T_i))$ be a tree-decomposition of $G_i$ with width at most $\tw(G)$. 
There is a vertex $v$ in some component $G_i$ and there is a vertex $w$ in some distinct component $G_j$ such that $v$ and $w$ are on a common face $f$. Embed the edge $vw$ across $f$. Let $T$ be obtained from the disjoint union $T_1\cup\dots\cup T_c$ by adding a new node $z$ with $B_z:=\{v,w\}$, and adding edges $xz$ and $xyz$ where $v\in B_{x}$ and $w\in B_{y}$. Add edges to the underlying forest to create a tree. We obtain a tree-decomposition of $G+vw$ with width $\tw(G)$. Now assume that $G$ is connected. 

Suppose $G$ has a cut-vertex $v$. So $G$ has subgraphs $G_1\subsetneq G$ and $G_2\subsetneq G$ with $G_1\cup G_2=G$ and $V(G_1\cap G_2)=\{v\}$. So $\tw(G_1)\leq\tw(G)$ and $\tw(G_2)\leq \tw(G)$. Since $G$ is connected, there are edges $uv\in E(G_1)$ and $vw\in E(G_2)$ with $u,v,w$ consecutive on some face $f$ of $G$. Since $v$ is a cut-vertex, $uw\not\in E(G)$. Embed the edge $uw$ across $f$. Let $(B_x:x\in V(T_i))$ be a tree-decomposition of $G_i$ with width at most $\tw(G)$, where $T_1$ and $T_2$ are disjoint. Let $x\in V(T_1)$ such that $u,v\in B_x$. Let $y\in V(T_2)$ such that $v,w\in B_y$. Let $T$ be obtained from $T_1\cup T_2$ by adding a new node $z$ adjacent to $x$ and $y$. With $B_z:=\{u,v,w\}$, we obtain a tree-decomposition of $G+uw$ with width at most $\max\{\tw(G),2\}$. Now assume that $G$ is 2-connected. 

Let $(B_x:x\in V(T))$ be a tree-decomposition of $G$ with width $\tw(G)$. Let $f$ be a non-triangular face of $G$. Since $G$ is 2-connected, $f$ is bounded by a cycle $C$. Let $H$ be the graph with $V(H):=V(C)$ where distinct vertices $v,w\in V(H)$ are adjacent if and only if there is a node $x\in V(T)$ such that $v,w\in B_x$. By construction, $(B_x\cap V(H):x\in V(T))$ is a tree-decomposition of $H$. Every edge of $C$ is in $H$. Consider an edge $vw$ of $H$ where $vw\not\in E(C)$. If $vw\not\in E(G)$ then $(B_x:x\in V(T))$ is a tree-decomposition of $G+vw$ with width $\tw(G)$, and we are done. So assume that $E(H)\subseteq E(G)$. Each edge in $E(H)\setminus E(C)$ is embedded outside of $f$. Consider the drawing of $H$ where each edge in $E(H)\setminus E(C)$ is drawn across $f$. If two edges cross in this drawing of $H$, then the corresponding edges cross in $G$ (outside of $f$). So no two edges cross in this drawing of $H$. In fact, we claim that $H$ is a triangulation of $C$ (a maximal outerplanar graph). If not, then there is an internal face of $H$ bounded by an induced cycle $D$ with at least four vertices. By construction, $(B_x\cap V(D):x\in V(T))$ is a tree-decomposition of $D$. Since $|V(D)|\geq 4$, $D$ is not chordal. A graph is chordal if and only if it has  a tree-decomposition in which each bag is a clique~\citep[Proposition~12.3.6]{Diestel5}. So $B_x\cap V(D)$ is not a clique for some $x\in V(T)$. Thus two non-adjacent vertices in $D$ are in $B_x$, which contradicts the definition of $H$. Thus $H$ is a triangulation of $C$. Hence there are vertices $a,b,c,d\in V(H)$ with 
$abc$ and $bcd$ triangular faces of $H$. So $ad\not\in E(G)$, as otherwise it would cross $bc$. The edge $bc$ in $G$ and the edge $bc$ in $H$ form a noose $\lambda$ in $G$ separating $a$ and $d$. Without loss of generality, $a$ is inside $\lambda$ and $d$ is outside $\lambda$. Let $G_1$ be the subgraph of $G$ induced by $b$, $c$ and every vertex inside $\lambda$. Note that $\tw(G_1)\leq\tw(G)$, and $abc$ is a triangle in $G_1$. Let $G_2$ be the subgraph of $G$ induced by $b$, $c$ and every vertex outside $\lambda$. Note that $\tw(G_2)\leq\tw(G)$ and $bcd$ is a triangle in $G_2$. For $i\in\{1,2\}$, let $(B_x:x\in V(T_i))$ be a tree-decomposition of $G_i$ with width at most $\tw(G)$, where $T_1$ and $T_2$ are disjoint. 
Let $x\in V(T_1)$ such that $a,b,c\in B_x$. 
Let $y\in V(T_2)$ such that $b,c,d\in B_y$. 
Let $T$ be obtained from $T_1\cup T_2$ by adding a new node $z$, and let $B_z:=\{a,b,c,d\}$. Observe that $(B_x:x\in V(T))$ is a tree-decomposition of $G+ad$ with width $\max\{\tw(G),3\}$. Embed $ad$ across $f$. 
\end{proof}

Note that \citet[Lemma~1]{BV13} proved \cref{embedding_preserving_triangulation} without the property that $G'$ preserves the embedding of $G$. We need the stronger property for the  proof of the next lemma.


\begin{lem}\label{curve_separator}
For any plane graph $G$ and face-weighting  $w:F(G)\to\mathbb{R}_{\ge 0}$ with total weight $N:=w(F(G))$, there exists a noose $\lambda$ of $G$ such that $|\lambda\cap V(G)|\le \max\{4\tw(G),12\}+4$ and $w(F(G)\cap F(C) )\le \frac23 N$ for each component $C$ of $G-\lambda$.
\end{lem}

\begin{proof}
By \cref{embedding_preserving_triangulation}, edges can be added to 
$G$ to obtain a plane triangulation $G'$ with $\tw(G')\leq\max\{\tw(G),3\}$. Since $G'$ is embedding-preserving, each face $f\in F(G)$ corresponds to a non-empty set $S_f\subseteq F(G')$, such that $S_{f_1}\cap S_{f_2}=\emptyset$ for distinct $f_1,f_2\in F(G)$. For each $f\in F(G)$ let $w(f'):=w(f)/|S_f|$ for each $f'\in S_f$. So the total weight of faces in $S_f$ equals $w(f)$, and the total weight of faces in $G'$ equals $N$. By \cref{cycle_separator}, $G'$ contains a cycle $\lambda$ of length at most $4\gm(G')+4\leq 4\tw(G')+4 \leq \max\{4\tw(G),12\}+4$ such that $w(F(G')\cap F(C')) \le \frac23 N$ for each component $C'$ of $G'-V(\lambda)$.  Since $G'$ is a supergraph of $G$, any component $C$ of $G-V(\lambda)$ is contained in some component $C'$ of $G-V(\lambda)$. 
Since $G'$ is embedding-preserving, any face in $F(G)\cap F(C)$ is a union of some faces in $F(G')\cap F(C')$, so $w( F(G)\cap F(C) )\le w( F(G')\cap F(C') ) \le \frac23 N$.  The result follows by treating the cycle $\lambda$ in $G'$ as a noose of $G$.
\end{proof}

\begin{lem}
\label{new_apply_curve_separator}
For any real numbers $N\ge W\ge 1$ and $t\ge 12\sqrt{10\log_{3/2} N + 36}+7$, for any plane graph $G$ and any face-weighting $w:F(G)\to\mathbb{R}_{\geq 0}$ such that $w(F(G))  \le N$, and  $w(F(G,v)) \ge W$ for each $v\in V(G)$, there exists $S\subseteq V(G)$ such that $|S|\le 
\frac{4778 N}{W(t-7)}+\frac{20480 N}{W(t-7)^2}$ and $\tw(G-S)< t$.
\end{lem}

\begin{proof}
For each connected subgraph $C$ of $G$, let $N_C:=w( F(C)\cap F(G) )$.  

Construct the set $S$ iteratively, where initially $S$ is empty.  At each step, if $\tw(G-S)< t$, we stop.  Otherwise, let $C$ be a component of $G-S$ with $\tw(C)\ge t$ and let $X_C:=F(C)\setminus F(G)$. Note that $w(f)$ is well-defined for each $f\in F(C)\setminus X_C=F(C)\cap F(G)$.  Furthermore, for each vertex $v\in V(C)$, $F(C,v)$ includes a face in $X_C$ or $w( F(C,v) )=w( F(G,v)) \ge W$. Observe that, $C$, $N_C$, and $X_C$ are compatible with \cref{tw_vs_weight}. Thus 
    \begin{equation}
    \label{TWbound}
     t  \leq \tw(C) \leq 12\sqrt{|X_C|+N_C/W} +7 .
    \end{equation}

Let $w(f):=0$ for each $f\in X_C$. Apply \cref{curve_separator} to $C$. We obtain a noose $\lambda_C$ of $C$ such that if $S_C:=\lambda_C\cap V(C)$ then 
    \begin{align}
    |S_C| 
     \leq \max\{4\tw(C),12\}+4
    & \leq \max\{48\sqrt{|X_C|+N_C/W}+28,12\}+4 \nonumber\\
    & = 48\sqrt{|X_C|+N_C/W}+32.\label{SizeSC}
    \end{align}
    and for every component $C'$ of $C-S_C$, the set of faces of $C'$ that are faces of $C$ has total weight at most $\frac23 N_C$. Every face in $F(C')\cap F(G)$ is also a face in $F(C)\cap F(G)$. So $N_{C'}\leq \frac23 N_C$. Add $S_C$ to the set $S$. 

    By construction, $\tw(G-S)< t$. Let $\mathcal{C}$ be the collection of components $C$ considered during the algorithm (for which $S_C$ is added to $S$). Consider $C\in\mathcal{C}$. Let $C'$ be a component of $C-S_C$. So $N_{C'}\le \frac23 N_C$. This implies the algorithm terminates. Furthermore, $X_{C'}\setminus X_C$ has at most one face, which contains $\lambda_C$. Therefore $|X_{C'}|\leq |X_{C}|+1$. By induction, $N_C \leq (\frac23)^{|X_C|} N$. Thus $|X_C| \leq \log_{3/2}(N/N_C)\leq \log_{3/2} N$. 

By assumption, $t-7 \geq 12\sqrt{36+10\log_{3/2} N} \geq 12\sqrt{36+10 |X_C|} $ for each $C\in\mathcal{C}$. By \cref{TWbound}, 
\begin{equation}
\label{some-equation}
\tfrac{N_C}{W} \geq (\tfrac{t-7}{12})^2 -|X_C| \geq 36+ 9|X_C|.
\end{equation}
For each integer $i\geq 0$, let 
\begin{equation*}
\mathcal{C}_i:=\{ C\in\mathcal{C}: (\tfrac{t-7}{12})^2 (\tfrac43)^i \leq |X_C| + \tfrac{N_C}{W} < (\tfrac{t-7}{12})^2 (\tfrac43)^{i+1} \}.
\end{equation*}
Thus, for each $C\in\mathcal{C}_i$, 
\begin{equation}
\label{range}
(t-7) (\tfrac43)^{i/2} \leq 
12 \sqrt{|X_C| + \tfrac{N_C}{W}} < 
(t-7) (\tfrac43)^{(i+1)/2}.
\end{equation}
By definition, $\mathcal{C}_0,\mathcal{C}_1,\dots$ are pairwise-disjoint. 
By \cref{TWbound}, each $C\in\mathcal{C}$ is in $\mathcal{C}_i$ for some $i\geq 0$. So $\{\mathcal{C}_0,\mathcal{C}_1,\dots\}$ is a partition of $\mathcal{C}$.  

We now analyse $|\mathcal{C}_i|$ for each integer $i\geq 0$. Consider $C\in\mathcal{C}_i$. 
By \cref{some-equation}, $\tfrac{N_C}{W} \geq 36+ 9|X_C|> 12+3|X_C|$, which can be rewritten as  $\tfrac{N_C}{W} +|X_C| >  \tfrac43(1+|X_C| + \tfrac23 \tfrac{N_C}{W})$. Thus, if $C'$ is a component of $C-S_C$, then
$$\tfrac{N_C}{W} +|X_C| >     \tfrac43(1+|X_C| + \tfrac23 \tfrac{N_C}{W})
     \geq  \tfrac43 (|X_{C'}| + \tfrac{N_{C'}}{W}).$$
Hence, if $C'\in\mathcal{C}_j$, then $j\geq i+1$. 
More generally, if $C'\in\mathcal{C}_j$ is any proper subgraph of $C$, then $j\geq i+1$. Thus $F(C)\cap F(D)\cap F(G)=\emptyset$ for distinct $C,D\in\mathcal{C}_i$. Hence, 
\begin{equation*}
\sum_{C\in \mathcal{C}_i} N_C
= 
\sum_{C\in \mathcal{C}_i} w(F(C)\cap F(G))
\le w(F(G))
\le N.
\end{equation*}
Note that \cref{some-equation} implies $\tfrac{N_C}{W} > \tfrac{9}{10} ( \tfrac{N_C}{W}+|X_C|)$. Hence, 
\begin{equation*}
    \sum_{C\in \mathcal{C}_i} \tfrac{9}{10} (\tfrac{t-7}{12})^2 (\tfrac43)^i
\leq 
\sum_{C\in \mathcal{C}_i} \tfrac{9}{10} ( \tfrac{N_C}{W} +|X_C| )
\leq 
\sum_{C\in \mathcal{C}_i} \tfrac{N_C}{W}
\le  \tfrac{N}{W} .
\end{equation*}
Therefore,
\begin{equation*}
|\mathcal{C}_i| 
\le \tfrac{160 N}{W (t-7)^2}  (\tfrac34)^i  .
\end{equation*}

We now analyse $|S|$. For each integer $i\geq 0$, by \eqref{SizeSC} and \eqref{range},
  \begin{align*}
      \sum_{C\in\mathcal{C}_i} |S_C| 
    \leq \sum_{C\in\mathcal{C}_i} 48 \sqrt{ |X_C| + N_C/W}+32
&     <|\mathcal{C}_i| \, \big(  4(t-7) (\tfrac43)^{(i+1)/2} +32\big)\\
 &    \leq \tfrac{160 N}{W (t-7)^2} (\tfrac34)^i 
 \big(  4(t-7) (\tfrac43)^{(i+1)/2} +32\big)\\
  &   \leq \tfrac{640 \sqrt{4/3}\, N}{W (t-7)} (\tfrac34)^{i/2} 
  + \tfrac{5120 N}{W (t-7)^2 } (\tfrac34)^i   .
    \end{align*}  
 Therefore
  \begin{align*}
    |S| =  
    \sum_{i\geq 0} \sum_{C\in\mathcal{C}_i} |S_C| 
  &  \leq
    \sum_{i\geq 0} \left( \tfrac{640 \sqrt{4/3}\, N}{W (t-7) }(\tfrac34)^{i/2}  
  + \tfrac{5120 N}{W (t-7)^2 } (\tfrac34)^i \right)\\
&    \leq\tfrac{640  \times 2(2+\sqrt{3}) N}{W (t-7)}
     + \tfrac{5120  \times 4 N}{W (t-7)^2}\\
 &   <\tfrac{4778 N}{W(t-7)}
    +\tfrac{20480 N}{W(t-7)^2}    .
    \qedhere
\end{align*}  
\end{proof}

We now complete the proof of our main result; refer to \cref{overview}.

\begin{thm}
\label{3ColourPlanarPrecise}
For sufficiently large $n$ every $n$-vertex planar graph $G$ is 3-colourable with clustering $16n^{4/9}$.
\end{thm}

\begin{proof}
Fix a plane embedding of $G$. By \cref{LT}, there is a $(16n^{4/9})$-separator $S_0$ of $G$ such that $|S_0| \leq \frac{12n}{(16 n^{4/9})^{1/2}} = 3 n^{7/9}$. Let $S\subseteq S_0$ be a vertex-minimal $(16n^{4/9})$-separator. So $|S| \leq 3n^{7/9}$. Colour every vertex in $G-S$ red. No other vertices will be coloured red. So every red component has at most $16 n^{4/9}$ vertices. 

Let $G_1$ be the plane subgraph of $G$ induced by $S$. For each face $f$ of $G_1$, let $w(f)$ be the number of vertices of $G-S$ embedded in the interior of $f$. So $w(F(G_1)) = |V(G-S)|\leq n$. 
For each vertex $v$ of $S$, by the minimality of $S$, there is a connected subgraph $C$ of $G-(S\setminus\{v\})$ with more than $16 n^{4/9}$ vertices. Each vertex in $C-v$ is embedded in the interior of a single face of $G_1$ incident to $v$. Thus 
$w(F(G_1,v)) \geq 16 n^{4/9}-1$. Hence \cref{new_apply_curve_separator} is applicable to $G_1$ with $N=n$ and $W=16 n^{4/9}-1$, where $t=40 n^{1/9} \geq 12\sqrt{10 \log_{3/2} N + 36}+7$. 
Thus there exists $S_1\subseteq V(G_1)$ such that $\tw(G_1-S_1)\leq t$ and 
\begin{align*}
    |S_1|
\leq \tfrac{4778n}{( 16 n^{4/9} -1) (40 n^{1/9}-7)} + 
\tfrac{20480 n }{ ( 16 n^{4/9} -1) (40 n^{1/9}-7)^2}
< 8 n^{4/9}.
\end{align*}
Colour every vertex in $S_1$ yellow. 

Let $G_2:= G_1-S_1$. \cref{TreewidthSep} is applicable to $G_2$ with $k:=\ceil{t-1}$ and $p=8n^{4/9}$ and $q=16n^{4/9}$, since $|V(G_2)|(k+1) \leq |S|\,t \leq (3 n^{7/9}) (40 n^{1/9}) < (8n^{4/9}) (16 n^{4/9})$ and $p=8n^{4/9} \geq 40n^{1/9}+1 \geq k+1$. We obtain a $(16n^{4/9})$-separator $S_2$ of $G_2$ of size at most $8 n^{4/9}$. Colour every vertex in $S_2$ yellow. The number of yellow vertices is at most $2\cdot 8n^{4/9}\leq 16n^{4/9}$. Colour every vertex in $G_2-S_2$ blue. Each blue component has at most $16n^{4/9}$ vertices. 

Therefore, $G$ is 3-coloured with clustering at most $16n^{4/9}$.
\end{proof}

\begin{figure}[!h]
\centering
\includegraphics{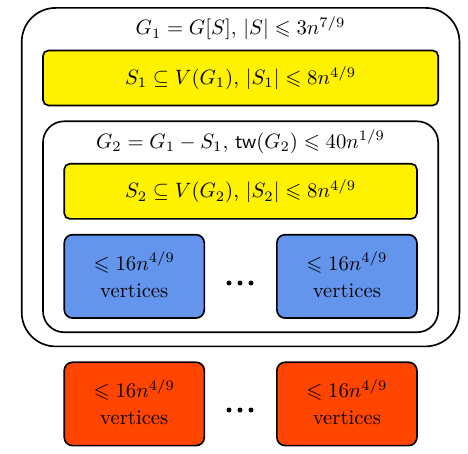}
\caption{Structure of the main proof.}
\label{overview}
\end{figure}

\subsection*{Acknowledgements}

This research was initiated at the  Barbados Graph Theory Workshop held at the Bellairs Research Institute in March 2025. Thanks to the other workshop participants for creating a productive working atmosphere. 

{\fontsize{10pt}{11pt}
\selectfont
\bibliographystyle{DavidNatbibStyle}
\bibliography{DavidBibliography}}

\end{document}